\documentclass[]{amsart}
\usepackage[utf8]{inputenc}
\usepackage{graphicx}
\usepackage{multirow}
\usepackage[english]{babel}
\usepackage{amsrefs}
\graphicspath{ {images/} }
\usepackage[T1]{fontenc}
\usepackage{txfonts}
\usepackage{times}
\usepackage[all]{xy}
\usepackage{subfig}
\usepackage{tikz}
\usetikzlibrary{shapes,arrows,shadows}
\usetikzlibrary{decorations.markings}
\usepackage{color}
\usepackage[normalem]{ulem}
\usepackage{hyperref}
\usepackage{wrapfig}
\usetikzlibrary{arrows}
\usepackage{pb-diagram}
\usepackage{verbatim}
\usepackage{float}
\usepackage{mathtools}
\usetikzlibrary{positioning}

\newtheorem{theorem}{Theorem}[section]

\newtheorem{proposition}{Proposition}[section]
\newtheorem{corollary}{Corollary}[section]

\newtheorem{example}{Example}[section]
\newtheorem{remark}{Remark}[section]
\newtheorem{question}{Question}[section]

\title{Dessins d'enfants and some holomorphic structures on the Loch Ness Monster}
\author[Y. Atarihuana, J. C. Garc\'{\i}a, R. A. Hidalgo, S. Quispe, C. Ram\'{\i}rez Maluendas]{Yasmina Atarihuana, Juan Garc\'{\i}a, Rub\'en A. Hidalgo, Sa\'ul Quispe, Camilo Ram\'{\i}rez Maluendas}
\address{Facultad de Ciencias, Universidad Central del Ecuador, Quito, Ecuador}
\email{yfatarihuana@uce.edu.ec, jcgarcian@uce.edu.ec}
\address{Departamento de Matem\'atica y Estad\'{\i}stica, Universidad de La Frontera, Temuco, Chile}
\email{ruben.hidalgo@ufrontera.cl, saul.quispe@ufrontera.cl}
\address{Departamento de Matem\'atica y Estad\'istica, Universidad Nacional de Colombia, Sede Manizales. Manizales, Colombia}
\email{camramirezma@unal.edu.co}
\thanks{Partially supported by Projects FONDECYT 1190001 and 11170129, and Project HERMES 47019}


\begin{document}

\begin{abstract}
The classical theory of dessin d'enfants, which are bipartite maps on compact orientable surfaces, are combinatorial objects used to study branched covers between compact Riemann surfaces and the absolute Galois group of the field
of rational numbers. In this paper, we show how this theory is naturally extended to non-compact orientable surfaces and, in particular, we observe that  the Loch Ness monster (the surface of infinite genus with exactly one end) admits infinitely many regular dessins d'enfants (either chiral or reflexive).  In addition, we study different holomorphic structures on the Loch Ness monster, which come from homology covers of compact Riemann surfaces, infinite hyperelliptic and infinite superelliptic curves. 
\end{abstract}

\maketitle

\section{Introduction}
In this paper, a surface will mean a (possible non-compact) second countable, connected and orientable $2$-manifold without boundary. A Riemann surface structure on a surface corresponds to a (maximal) holomorphic atlas (its local charts take their values in the complex plane $\mathbb{C}$ and have biholomorphic transition functions where their domains overlap). We will use $X$ to denote a surface and $S$ to denote a Riemann surface.

A dessin d'enfant corresponds to a bipartite map on a compact surface. These objects were studied as early as the nineteenth century and rediscovered by Grothendieck in the twentieth century in his ambitious research outline \cite{Gro}. 
Each dessin d'enfant defines (up to isomorphisms) a unique Riemann surface structure $S$ together with a non-constant holomorphic branched cover $\beta:S \to \widehat{\mathbb C}$ whose branched value set is a subset of $\{\infty,0,1\}$ (called a Belyi map on $S$). Conversely, every Belyi map $\beta:S \to \widehat{\mathbb C}$ comes from a suitable dessin d'enfant.

The goal of this paper is to focus on the generalization of the theory of dessins d'enfants to non-compact surfaces, and the study of certain holomorphic structures on the Loch Ness monster (LNM) generalizing cyclic gonal curves (hyperelliptic, superelliptic).

From the point of view of Ker\'ekj\'art\'o's classification theorem of non-compact surfaces (see e.g., \cites{Ker, Ian}), the topological type of any surface $X$ (recall that we are assuming no boundary) is given by: (i) its genus $g\in \mathbb{N}\cup \{\infty\}$ and (ii) a couple of nested, compact, metrizable and totally disconnected spaces ${\rm Ends}_{\infty}(X)\subset {\rm Ends}(X)$. The set ${\rm Ends}(X)$ (respectively, ${\rm Ends}_{\infty}(X)$) is known as the ends space (respectively, the non-planar ends space) of $X$. Of all non-compact surfaces we focus is the LNM, which is the unique infinite genus surface with exactly one end.

In Section 2, we collect the principal tools to understand the classification of non-compact surfaces theorem. Also, we explore the concept of ends on groups.
In Sections 3, 4 and 5, we provide an introduction to the theory of dessins d'enfants, Belyi pairs for a general surface (compact or not) and a characterisation of dessins d'enfants in terms of Belyi pairs and triangle groups. In Section 6, we describe infinitely many regular dessins d'enfants (either chiral or reflexive) on the LNM, by observing that it is the homology cover of any compact Riemann surface of genus $g \geq 2$.  We give a precise description of an infinite set of generators of a Fuchsian group $\Gamma$, such that the quotient $\mathbb{H}/\Gamma$ is the LNM. We finally discuss on the explicit representation of the LNM as a smooth infinite hyperelliptic and superelliptic curves. For the infinite hyperelliptic curves we also describe the corresponding moduli space.

%%%%%%%%%%%%%%%%%%%%%%%%%%%%%%%%%%%%%%%%%%%%%%%%%%%%%%%%%%%%%%
%%%%%%%%%%%%%%%%%%%%%%%%%%%%%%%%%%%%%%%%%%%%%%%%%%%%%%%%%%%%%%

\section{The ends space} 
In this section, we shall recall the concepts of ends of topological spaces, surfaces and groups (see \cite{Fre}*{1. Kapitel}).

\subsection{The ends space of a topological space}\label{subsection:ends}  Let $Z$ be a locally compact, locally connected, connected, and Hausdorff space. 

A sequence $(U_n)_{n\in\mathbb{N}}$, of non-empty connected open subsets of $Z$, is called {\it nested} if the following conditions hold:
\begin{enumerate}
\item $U_{1}\supset U_{2}\supset\ldots$

\item  for each $n\in\mathbb{N}$ the boundary $\partial U_n$ of $U_{n}$ is compact, 

\item the intersection $\cap_{n\in\mathbb{N}}\overline{U_{n}}=\emptyset
$, and

\item for each compact  $K\subset Z$ there is $m\in\mathbb{N}$ such that $K\cap U_m =\emptyset$.
\end{enumerate}

Two nested sequences $(U_n)_{n\in\mathbb{N}}$ and $(U'_{n})_{n\in\mathbb{N}}$ are {\it equivalent} if for each $n\in\mathbb{N}$ there exist $j,k\in\mathbb{N}$ such that $U_{n}\supset U'_{j}$, and $U'_{n}\supset U_{k}$. We denote by the symbol $[U_{n}]_{n\in\mathbb{N}}$ the equivalence class of $(U_n)_{n\in\mathbb{N}}$ and we called it an {\it end} of $Z$.  

The {\it ends space} of $Z$ is the topological space, having the ends of $Z$ as elements, and endowed with the topology generated by the basis
\begin{equation*}
U^{*}:=\{[U_{n}]_{n\in\mathbb{N}}\in{\rm Ends}(Z)\hspace{1mm}|\hspace{1mm}U_{j}\subset U\hspace{1mm}\text{for some }j\in\mathbb{N}\},
\end{equation*}
where $U \subset Z$ is a non-empty open subset whose boundary $\partial U$ is compact.

\begin{theorem}[\cite{Ray}*{Theorem 1.5}]
The ends space ${\rm Ends}(Z)$, with the above topology, is Hausdorff, totally disconnected,
and compact.
\end{theorem}

%%%%%%%%%%%%%%%%%%%%%%%%%%%%%%%%%%%%
\subsection{Ends of a surface} 
Let us now consider a surface $X$ and its ends space.

By a \emph{subsurface} of $X$ we mean an embedded surface, which is a closed subset of $X$ and whose boundary consists of a finite number of nonintersecting simple closed curves. Note that a subsurface it might or not be compact.
The {\it reduced genus} of a compact subsurface $Y$,  with $q(Y)$ boundary curves and Euler characteristic $\chi(Y)$, is the number $g(Y)=1-\frac{1}{2}(\chi(Y)+q(Y))$. 
The {\it genus} of the surface $X$ is the supremum of the genera of its compact subsurfaces; so it can be an non-negative integer or $\infty$. 
The surface $X$ is said to be {\it planar} if it has genus zero (in other words, $X$ is homeomorphic to an open of the complex plane). 

\begin{remark}
In this case,  in the definition of ends given in section \ref{subsection:ends}, we may assume the sequence $(U_{n})_{n\in\mathbb{N}}$, such that the clousures $\overline{U}_{n}$ are subsurfaces. In this setting, an end $[U_n]_{n\in\mathbb{N}}$ of a surface $X$ is called {\it planar} if there is $l\in\mathbb{N}$ such that the subsurface $\overline{U}_l\subset X$ is planar.
\end{remark}

 Hence, we define the subset ${\rm Ends}_{\infty}(X)$ of ${\rm Ends}(X)$, conformed by all ends of $X$, which are not planar ({\it ends having infinity genus}). It follows directly from the definition that ${\rm Ends}_{\infty}(X)$ is a closed subset of ${\rm Ends}(X)$ (see \cite{Ian}*{p. 261}) and that $(g,{\rm Ends}_{\infty}(X),{\rm Ends}(X))$, where $g$ is the genus of $X$, is a topological invariant.

\begin{theorem}[Classification of non-compact surfaces, \cite{Ker}*{\S 7, Erster Abschnitt}, \cite{Ian}*{Theorem 1}]
Two surfaces $X_1$ and $X_2$  having the same genus are topological equivalent if and only if there exists a homeomorphism $f: {\rm Ends}(X_1)\to {\rm Ends}(X_2)$ such that $f( {\rm Ends}_{\infty}(X_1))= {\rm Ends}_{\infty}(X_2)$.
\end{theorem}

%%%%%%%%%%%%%%%%%%%%%%%%%%%%%%%
\subsection{Loch Ness monster (\cite{Val}*{Definition 2})}\label{d:loch_nesss_monster}
 The {\it Loch Ness monster}\footnote{From the historical point of view as shown in \cite{AyC}, this
nomenclature is due to A. Phillips and D. Sullivan \cite{PSul}.} (LNM) is the unique, up to homeomorphisms, surface of infinite genus and exactly one end (see Figure \ref{F:LMN1}).
 \begin{figure}[h!]
\begin{center}
\includegraphics[scale=0.3]{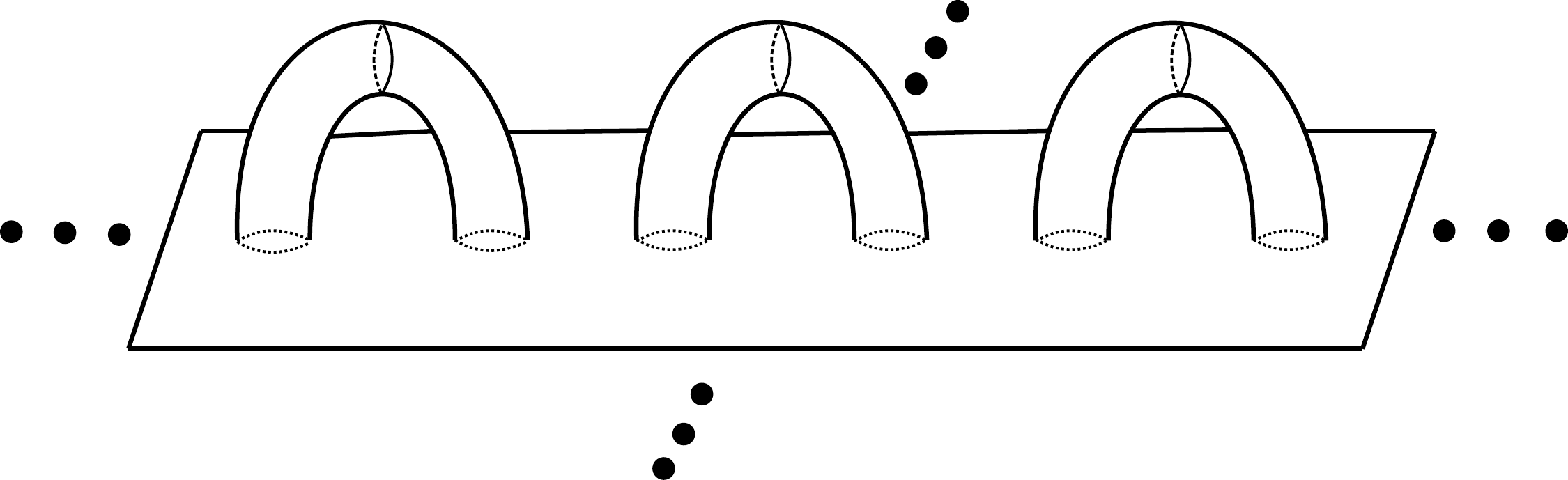}\\
  \caption{\emph{The Loch Ness Monster}}
   \label{F:LMN1}
\end{center}
\end{figure}

\begin{remark}[\cite{SPE}*{\S 5.1., p. 320}] The surface $X$ has one end if and only if for all compact subset $K \subset X$ there is a compact $K^{'}\subset X$ such that $K\subset K^{'}$ and $X\setminus  K^{'}$ is connected.
\end{remark}

%%%%%%%%%%%%%%%%%%%%%%%%%%%%%%%%%%
\subsection{Ends of a group} 
Let $G$ be a finitely generated group and let ${\mathcal C} \subset G$
be a finite generating set (closed under inversion).  The {\it Cayley graph of $G$, with respect to ${\mathcal C}$}, is the graph ${\rm Cay}(G,{\mathcal C})$ whose vertices are the elements of $G$ and, there is an edge with ends points $g_1$ and $g_2$ if and only if there is an element $h\in {\mathcal C}$ such that $g_1h=g_2$. 

It is known that the Cayley graph\footnote{For us ${\rm Cay}(G,{\mathcal C})$ will be the geometric realization of an abstract graph (see \cite{Diestel2}*{p.226}).} ${\rm Cay}(G,{\mathcal C})$ is locally compact, locally connected, connected, Hausdorff space. In particular, one may consider the ends space of ${\rm Cay}(G,{\mathcal C})$.

\begin{proposition}[\cite{Loh}*{Proposition 8.2.5}] 
Let $G$ be a finitely generated group. The ends space of the Cayley graph of $G$ does not depend on the choice of the finite genereting set.
\end{proposition}

As a consequence of the above fact, one may define 
the {\it ends space of $G$} as ${\rm Ends}(G):={\rm Ends}({\rm Cay}(G,{\mathcal C}))$.

\begin{theorem}[\cite{ScottWall}*{Lemmas 5.6, 5.7, and Corollary 5.9}]
Let $G$ be a finitely generated group. Then
\begin{enumerate}
    \item $G$ has either zero, one, two or infinitely many ends.
    \item If $K$ is a finite index subgroup of $G$, then both groups have the same number of ends.
    \item If $K$ is a finite normal subgroup of $G$, then groups $G$ and $G/K$ have the same number of ends.
\end{enumerate}
\end{theorem}

The following result asserts that $G$ has more than one end if and only if $G$ splits over a finite subgroup.

\begin{theorem}[\cite{Stallings68}*{p. 312}, \cite{Stallings71}*{\S 4.A.6, p. 38}, \cite{ScottWall}*{Theorem 6.1}]\label{unnfinal}
Let $G$ be a finitely generated group.
\begin{enumerate}
\item If $G$ has infinitely many ends, then one of the following hold.
\begin{enumerate}
    \item If $G$ is torsion-free, then $G$ is  a non-trivial free product.
    \item If $G$ has torsion, then $G$ is a non-trivial free product with amalgamation, with finite amalgamated subgroup.
\end{enumerate}
\item The following are equivalent.
\begin{enumerate}
    \item $G$ has two ends.
    \item $G$ has a copy of ${\mathbb Z}$ as a finite index subgroup.
    \item $G$ has a finite normal subgroup $N$ with $G/N$ isomorphic to either ${\mathbb Z}$ or ${\mathbb Z}_{2} * {\mathbb Z}_{2}$.
    \item Either $G=F*_{F}$, where $F$ is a finite group, or $G=A*_{F}B$, where $F$ is a finite group and $[A:F]=[B:F]=2$.
\end{enumerate}
\end{enumerate}
\end{theorem}

\begin{proposition}[\cite{Geoghegan}*{\v{S}varc-Milnor Lemma, p. 440}, \cite{Loh}*{Proposition 8.2.5}]
\label{pro:acts_properly}
If a finitely generated group acts properly and cocompactly on a proper geodesic metric space then the number of ends of this space is the same as the number of ends of the group.
\end{proposition}

%%%%%%%%%%%%%%%%%%%%%%%%%%%%%%%%%%%%%%%%%%%%%%%%%%%%%%%
%%%%%%%%%%%%%%%%%%%%%%%%%%%%%%%%%%%%%%%%%%%%%%%%%%%%%%
\section{Dessins d'enfants}
The classical theory of (Grothendieck's) dessins d'enfants corresponds to bipartite maps on compact connected orientable surfaces (see \cites{GiGo, J-W}). This theory can be carried out without problems to non-compact connected orientable surfaces and, in this section, we describe this in such a generality.

%%%%%%%%%%%%%%%%%%%%%%%%%%%%%%%%%%%
\subsection{Dessins d'enfant}
A {\it dessin d'enfant} is a tuple $${\mathcal D}=(X,\Gamma,\iota:\Gamma \hookrightarrow X),$$
where
\begin{enumerate}
\item $X$ is a secound countable connected (not necessarily compact) orientable surface;

\item $\Gamma$ is a connected bipartite graph (vertices are either black or white) such that every vertex has finite degree;

\item $\iota:\Gamma \hookrightarrow X$ is an embedding such that every connected component of $X\setminus \iota(\Gamma)$, called a {\it face} of the dessin d'enfant, is a polygon with a finite number (necessarily even) of sides. The {\it degree of a face} is half the number of its sides (where each side, which is the boundary of exactly one face,  has to be counted twice).
    
\end{enumerate}

Let us observe, from the above definition, that every compact subset $K$ of $X$ intersects only a finite number of faces, edges and vertices, respectively.

\medskip

We say that a dessin d'enfant ${\mathcal D}=(X,\Gamma,\iota:\Gamma \hookrightarrow X)$ is:
\begin{enumerate}
\item a {\it Grothendieck's dessin d'enfant} if $X$ is compact. In this case, we also say that it is a {\it dessin d'enfant of genus $g$}, where $g$ is the genus of $X$;

\item a {\it uniform dessin d'enfant} if all the black vertices (respectively, the white vertices and the faces) have the same degree;

\item a {\it bounded dessin d'enfant} if there is an integer $M>0$ such that the degrees of all vertices and faces are bounded above by $M$;

\item a {\it clean dessin d'enfant} if all white vertices have degree $2$. (This corresponds to the classical theory of maps on surfaces.)

\end{enumerate}

%%%%%%%%%%%%%%%%%%%%%%%%%%%%%%%%%%%%%%%
\subsection{Passport (valence) of a dessin d'enfant}
Let ${\mathcal D}=(X,\Gamma,\iota:\Gamma \hookrightarrow X)$ be a dessin d'enfant. The collections of vertices and faces of ${\mathcal D}$ are either finite or countable infinite. 

Let $\{v_{i}\}_{i \geq 1}$ (respectively, $\{w_{j}\}_{j \geq 1}$) be the collection of black (respectively, the collection of white) vertices of $\Gamma$, and let  $\{f_{k}\}_{k \geq 1}$ be the collection of faces of the dessin d'enfant. The {\it passport} (or {\it valence}) of ${\mathcal D}$ is the tuple
$$Val({\mathcal D})=(\alpha_{1}, \alpha_{2}, \ldots;\beta_{1},\beta_{2}, \ldots;\gamma_{1},\gamma_{2},\ldots),$$ where $\alpha_{i}$, $\beta_{j}$ and $\gamma_{k}$ denote the degrees of $v_{i}$, $w_{j}$ and $f_{k}$, respectively, with $\alpha_{i} \leq \alpha_{i+1}$, $\beta_{j} \leq \beta_{j+1}$ and $\gamma_{k} \leq \gamma_{k+1}$.

%%%%%%%%%%%%%%%%%%%%%%%%%%%%%%%%%%%%%%%%%
\subsection{Equivalence of dessins d'enfants}
Two dessins d'enfants
$${\mathcal D}_{1}=(X_{1},\Gamma_{1},\iota_{1}:\Gamma_{1} \hookrightarrow X_{1})\;  \mbox{ and } \; 
{\mathcal D}_{2}=(X_{2},\Gamma_{2},\iota_{2}:\Gamma_{2} \hookrightarrow X_{2})$$
are called {\it equivalent} if there exists 
an orientation-preserving homeomorphism $\phi:X_{1} \to X_{2}$ inducing an isomorphism of bipartite graphs (i.e., an isomorphism of graphs sending black vertices to black vertices).

\begin{remark}
Two equivalent dessins d'enfants necessarily have the same passport, but the converse is in general false.
\end{remark}

%%%%%%%%%%%%%%%%%%%%%%%%%%%%%%%%%%%%%%%%%%
\subsection{Automorphisms of dessins d'enfants}
Let ${\mathcal D}=(X,\Gamma,\iota:\Gamma \hookrightarrow X)$ be a dessin d'enfant.

An {\it orientation-preserving automorphism} (respectively, an {\it orientation-reserving automorphism}) of ${\mathcal D}$ is an automorphism $\rho$ of $\Gamma$, as a bipartite graph, such that there exists an orientation-preserving (respectively, orientation-reversing) homeomorphism $\phi:X \to X$ inducing it.

The group of orientation-preserving automorphisms of ${\mathcal D}$ is denoted by ${\rm Aut}^{+}({\mathcal D})$, and the group all automorphisms of ${\mathcal D}$, both orientation-preserving and orientation-reserving, is denoted by ${\rm Aut}({\mathcal D})$. The subgroup ${\rm Aut}^{+}({\mathcal D})$ has index at most two in ${\rm Aut}({\mathcal D})$.

If the index is two, then we say that ${\mathcal D}$ is {\it reflexive}; otherwise we say that it is {\it chiral}; and if ${\rm Aut}^{+}({\mathcal D})$ acts transitively on the set of edges of $\Gamma$, then we say that ${\mathcal D}$ is {\it regular}.

%%%%%%%%%%%%%%%%%%%%%%%%%%%%%%%%%%%%%%%%%%%
\subsection{Marked monodromy groups of dessins d'enfants}
Let ${\mathcal D}=(X,\Gamma,\iota:\Gamma \hookrightarrow X)$ be a dessin d'enfant. Let $E$ be the set of edges of $\Gamma$ and let ${\mathfrak S}_{E}$ be the permutation group of $E$.

Let $v_{i}$ be a black vertex of $\Gamma$, which has degree $\alpha_{i}$. If $e_{i1},\ldots,e_{i\alpha_{i}} \in E$ are the edges of $\Gamma$ adjacent to $v_{i}$, following the counterclockwise orientation of $X$, then we can construct a cyclic permutation 
$\sigma_{i}=(e_{i1},\ldots,e_{i\alpha_{i}}) \in {\mathfrak S}_{E}$. Then we consider 
the permutation $$\sigma=\prod_{i}\sigma_{i} \in {\mathfrak S}_{E}.$$ 
We may proceed in a similar fashion for the white vertices $w_{j}$ to construct a permutation 
$$\tau=\prod_{j}\tau_{j} \in {\mathfrak S}_{E}.$$

The {\it marked monodromy group} of ${\mathcal D}$ is the subgroup $M_{\mathcal D}=\langle \sigma, \tau \rangle$ generated by $\sigma$ and $\tau$ in the symmetric group ${\mathfrak S}_{E}$. If we are not interested in the generators, but just in the group, we talk of the {\it monodromy group}.

\begin{remark}
The connectivity of $\Gamma$ asserts that the marked monodromy group $M_{\mathcal D}$ is a transitive subgroup of ${\mathfrak S}_{E}$.
\end{remark}

The permutation $\tau\sigma$ is again a product of disjoint finite cycle permutations
$$\tau\sigma=\prod_{k}\eta_{k},$$
where there is a bijection between these $\eta_{k}$ and the faces $f_{k}$ of ${\mathcal D}$ (the length of $\eta_{k}$ is equal to the degree of the correspondent face $f_{k}$).

\begin{remark}[\cite{GiGo}*{Proposition 4.13}]
If $\sigma, \tau \in {\mathfrak S}_{E}$ are two permutations, such that $\sigma$, $\tau$ and $\tau\sigma$ are each a product of disjoint finite cycle permutation and the group generated by them is transitive, then $M=\langle \sigma, \tau \rangle$ is the marked monodromy group of some dessin d'enfant with $E$ as its set of edges.
\end{remark}

\begin{remark}
The group ${\rm Aut}^{+}({\mathcal D})$, for a regular dessin d'enfant, can be generated by two elements. 
\end{remark}

%%%%%%%%%%%%%%%%%%%%%%%%%%%%%%%%%%%%%
\subsection{The automorphisms of the dessin d'enfant in terms of the marked monodromy group}
Let ${\mathcal D}$ be a dessin d'enfant with marked monodromy group $M_{\mathcal D}=\langle \sigma, \tau\rangle< {\mathfrak S}_{E}$.
In terms of the monodromy group $M_{\mathcal D}$, the group ${\rm Aut}^{+}({\mathcal D})$ can be identified with the centralizer of $M_{\mathcal D}$ in ${\mathfrak S}_{E}$, that is, with the subgroup formed of those $\eta \in {\mathfrak S}_{E}$ such that $\eta \sigma \eta^{-1}=\sigma$ and $\eta \tau \eta^{-1}=\tau$.

Moreover, each orientation-reversing automorphism of the dessin can be identified with those $\eta \in {\mathfrak S}_{E}$ such that 
$\eta \sigma \eta^{-1}=\sigma^{-1}$ and $\eta \tau \eta^{-1}=\tau^{-1}$.

\vspace{3mm}

%%%%%%%%%%%%%%%%%%%%%%%%%%%%%
%%%%%%%%%%%%%%%%%%%%%%%%%%%%%
\section{Belyi pairs}
The classical theory of Belyi curves corresponds to compact Riemann surfaces which can be defined over the algebraic numbers. We extend this concept to the class of non-compact Riemann surfaces in order to relate them to dessins d'enfants as previously defined.

%%%%%%%%%%%%%%%%%%%%%
\subsection{Locally finite holomorphic branched coverings}
Let $S_1, S_2$ be connected Riemann surfaces and $\varphi:S_1 \to S_2$ be a surjective holomorphic map. We say that $\varphi$ is a {\it locally finite holomorphic branched cover map} if:
\begin{enumerate}
\item the locus of branched values $B_{\varphi} \subset S_2$ of $\varphi$ is a (which might be empty) discrete set; 
\item $\varphi:S_1 \setminus \varphi^{-1}(B_{\varphi}) \to S_2 \setminus B_{\varphi}$ is a holomorphic covering map; and
\item each point $q \in B_{\varphi}$ has an open connected neighborhood $U$ such that $\varphi^{-1}(U)$ consists of a collection $\{V_{j}\}_{j\in I}$ of pairwise disjoint connected open sets such that each of the restrictions $\varphi|_{V_j}:V_{j} \to U$  is a finite degree branched cover (i.e., is equivalent to a branched cover of the form $z \in {\mathbb D} \mapsto z^{d_{j}} \in {\mathbb D}$, where ${\mathbb D}$ denotes the unit disc).
\end{enumerate}

Observe that in the case the surface $S_1$ is compact, then any non-constant holomorphic map is a locally finite branched cover map. The above definition is needed for the non-compact situation.

%%%%%%%%%%%%%%%%%%%%%%%%%%%%%%%%%%%%%%%
\subsection{Belyi pairs}
Let $S$ be a connected Riemann surface (not necessarily compact).
A {\it Belyi map} on $S$ is a locally finite holomorphic branched cover map $\beta:S \to \widehat{\mathbb C}$ whose branch values are contained inside the set $\{\infty, 0, 1\}$. In this case, we also say that $S$ is a {\it Belyi surface} and that $(S,\beta)$ is a {\it Belyi pair}.

A Belyi pair $(S,\beta)$ is called a {\it bounded Belyi pair} if the set of local degrees of the preimeages of $0$, $1$ and $\infty$ is bounded; and is called {\it uniform} if the local degrees of $\beta$ at the preimages of each branched value is the same, but they might be different for different values.

Two Belyi pairs $(S_{1},\beta_{1})$ and $(S_{2},\beta_{2})$ are called {\it equivalent} if there is a biholomorphism $\phi:S_{1} \to S_{2}$ such that $\beta_{1}=\beta_{2} \circ \phi$.

%%%%%%%%%%%%%%%%%%%%%%%%%%%%%%%%%%%%%%%%%
\subsection{Automorphisms of Belyi pairs}
A {\it holomorphic automorphism} (respectively, an {\it antiholomorphic automorphism}) of a Belyi pair  $(S,\beta)$ is a conformal (respectively, anticonformal) automorphism $\phi$ of $S$ such that $\beta \circ \phi=\beta$ (respectively, $\beta \circ \psi=J\circ \beta$, where $J(z)=\overline{z}$). 

The group of all automorphisms of $(S, \beta)$, both holomorphic and antiholomorphic, is denoted by ${\rm Aut}(S, \beta)$ and its subgroup (of index at most two) consisting of 
the holomorphic ones is denoted by ${\rm Aut}^{+}(S, \beta)$.

The Belyi pair $(S,\beta)$ is called {\it regular} if $\beta$ is a regular branched covering, in which case its deck group is ${\rm Aut}^{+}(S,\beta)$.

%%%%%%%%%%%%%%%%%%%%%%%%%%%%%%%%%%%%%%
\subsection{Belyi pairs and dessins d'enfants}

\begin{theorem}[\cite{GiGo}*{Proposition 4.20}]

\label{Proposition:Belyi-dessins}
There is an one-to-one correspondence between the category of equivalence classes of (bounded, Grothendieck) dessins d'enfants and the equivalence classes of (bounded, finite degree) Belyi pairs. The correspondence preserves regularity and also uniformity.
\end{theorem}
\begin{proof}[Idea of the proof]
A dessin d'enfant ${\mathcal D}=(X,\Gamma,\iota:\Gamma \hookrightarrow X)$ induces a surjective continuous map $\beta:X \to \widehat{\mathbb C}$ which sends the black vertices to $0$, the white vertices to $1$ and center of faces to $\infty$ defining a covering map $\beta:X\setminus \beta^{-1}(\{\infty,0,1\}) \to \widehat{\mathbb C}\setminus\{\infty,0,1\}$. We may pull-back, under $\beta$, the Riemann surface structure of $\widehat{\mathbb C}\setminus \{\infty,0,1\}$ to obtain a Riemann surface structure on 
$X\setminus \beta^{-1}(\{\infty,0,1\})$ making the restriction of $\beta$ holomorphic. Such a Riemann surface structure extends to a Riemann surface structure $S$ on $X$ making $\beta$ a Belyi map. Conversely, if $(S,\beta)$ is a Belyi pair, then $\Gamma=\beta^{-1}([0,1])$ produces a dessin d'enfant on $S$.
\end{proof}

\begin{theorem}[Belyi's theorem \cite{Belyi80}*{Theorem 4}]
Every Belyi pair of finite degree (Grothendieck's dessin d'enfant) is equivalen to one of the form $(S,\beta)$, where $S$ is a smooth algebraic curve defined over $\overline{\mathbb Q}$ and with  $\beta$ a rational map also defined over it.
\end{theorem}

The previous results asserts that there is an action of the absolute Galois group $Gal(\overline{\mathbb Q}/\mathbb{Q})$ on Grothendieck's dessins d'enfants.

\begin{theorem}[\cites{GiGo1, Gro,Sch}]
The action of $Gal(\overline{\mathbb Q}/\mathbb{Q})$ on Grothendieck's dessins d'enfants is faithful.
\end{theorem}

%%%%%%%%%%%%%%%%%%%%%%%%%%%%%%%%%%%%%%%%%%%%
%%%%%%%%%%%%%%%%%%%%%%%%%%%%%%%%%%%%%%%%%%%%%
\section{Dessins d'enfants and triangular groups}

\subsection{Dessins subgroups of $\Gamma(2)$}
The aim of this section is to produce dessins d'enfants (or Belyi pairs) by suitable subgroups of 
$$\Gamma(2)=\langle A(z)=z+2, B(z)=z/(1-2z)\rangle \cong F_{2},$$ 
extending the known situation for the case of Grothendieck's dessins d'enfants.

Note that $\Gamma(2)$ is a normal subgroup of ${\rm PSL}_{2}({\mathbb Z})$ of index six such that ${\rm PSL}_{2}({\mathbb Z})/\Gamma(2) \cong {\mathfrak S}_{3}$.
The quotient ${\mathbb H}/\Gamma(2)$ is isomorphic to $\widehat{\mathbb C} \setminus \{\infty,0,1\}$.

Let $S$ be a connected Riemann surface and $\beta:S \to \widehat{\mathbb{C}}$ be a surjective meromorphic map, with branch values contained in $\{\infty,0,1\}$, which defines a covering map $\beta: S \setminus \beta^{-1}(\{\infty,0,1\}) \to \widehat{\mathbb C}\setminus \{\infty,0,1\}$. 
In this generality, it might happen that $\beta$ is not a locally finite branched cover map (so it might be not a Belyi map). 
By the covering maps theory, there is a subgroup $K$ of $\Gamma(2)$ such that the unbranched holomorphic cover $\beta:S\setminus \beta^{-1}(\{\infty,0,1\})\to \widehat{\mathbb C}\setminus \{\infty,0,1\}$ is induced by the inclusion $K<\Gamma(2)$.

The condition for $\beta$ to define a Belyi map on $S$ (i.e., to be a locally finite branched covering map) is equivalent for $K$ to satisfy the following property:
for every parabolic element $Z \in \Gamma(2)$ there is some positive integer $n_{Z}>0$ such that $Z^{n_{Z}}\in K$. We say that such kind of subgroup of $\Gamma(2)$ is a 
{\it dessin subgroup}. 

The following result states an equivalence between dessin subgroups of $\Gamma(2)$ and dessins d'enfants (Belyi pairs).

\begin{theorem}[\cite{GiGo}*{Theorem 4.31}]
There is an one-to-one correspondence between the category of equivalence classes of dessins d'enfants and congugacy classes of dessins subgroups of $\Gamma(2)$. In this equivalence, Grothendieck's dessins d'enfants corresponds to finite index subgroups.
\end{theorem}

%%%%%%%%%%%%%%%%%%%%%%%%%%%%%%
\subsection{Bounded Belyi pairs and triangular groups}
Let us consider a bounded Belyi pair $(S,\beta)$. The boundness condition permits to compute the least common multiple of all local degrees of the points in each fiber $\beta^{-1}(p)$, $p \in \widehat{\mathbb C}$.

Let $a,b,c \geq 1$ be the least common multiple of the local degrees of $\beta$ at the preimages of $0$, $1$ and $\infty$, respectively. In this case, the triple $(a,b,c)$ is called the {\it type} of $(S,\beta)$. By the equivalence with bounded dessins d'enfants, the above is also the type of the associated dessin d'enfant.

Set ${\mathbb X}(a,b,c)$ equal to the hyperbolic plane ${\mathbb H}$, the complex plane ${\mathbb C}$ or the Riemann sphere $\widehat{\mathbb C}$ if $a^{-1}+b^{-1}+c^{-1}$ is less than $1$, equal to $1$ or bigger than $1$, respectively. 

Let us consider a triangular group (unique up to conjugation by M\"obius transformations)
$\Delta(a,b,c)=\langle x,y: x^{a}=y^{b}=(yx)^{c}=1\rangle$, acting as a discontinuous group of holomorphic automorphisms of ${\mathbb X}(a,b,c)$. The quotient complex orbifold ${\mathcal O}(a,b,c):={\mathbb X}/\Delta(a,b,c)$ is the Riemann sphere whose cone points are $0$ (of cone order $a$), $1$ (of cone order $b$) and $\infty$ (of cone order $c$). 

By the uniformization theorem, there exists a subgroup $K$ of $\Delta(a,b,c)$ such that:
\begin{itemize}
\item[(i)] the quotient orbifold $S_{K}:={\mathbb X}(a,b,c)/K$ has a Riemann surface structure biholomorphically equivalent to $S$, and 
\item[(ii)] the Belyi map $\beta$ is induced by the inclusion $K<\Delta(a,b,c)$ (finite index condition on $K$ is equivalent for $S$ to be compact). 
\end{itemize}

\begin{theorem}[\cite{GiGo}*{Theorem 4.31}, \cite{J-W}*{Theorem 3.10}]\label{equiv}
There is a natural one-to-one correspondence between the category of equivalence classes of bounded Belyi pairs (bounded dessins d'enfants) of type $(a,b,c)$ and the category of conjugacy classes of subgroups of $\Delta(a,b,c)$. The finite degree Belyi pairs (Grothendieck's dessins d'enfants) correspond to finite index subgroups.

The regular Belyi pairs (regular dessins d'enfants) correspond to the torsion-free normal subgroups. The uniform ones correspond to torsion-free subgroups. 
\end{theorem}

 There is a group $$\overline{\Delta}(a,b,c)=\langle \tau_{1}, \tau_{2}, \tau_{3}: \tau_{1}^{2}=\tau_{2}^{2}=\tau_{3}^{2}=(\tau_{2}\tau_{1})^{a}=(\tau_{1}\tau_{3})^{b}=(\tau_{3}\tau_{2})^{c}=1\rangle,$$
where $\tau_{1},\tau_{2}$ and $\tau_{3}$ are reflection on the three sides of a circular triangle with angles $\pi/a,\pi/{b}, \pi/c$, such that $\Delta(a,b,c)$ is its index two subgroup of orientation-preserving elements ($x=\tau_{2}\tau_{1}, y=\tau_{1}\tau_{3}$). 

\begin{theorem}[\cite{GiGo}*{Theorem 4.43}, \cite{J-W}*{\S 3.3.2}]\label{autgroups}
Let $K<\Delta(a,b,c)$ and $(S,\beta)$ a Belyi pair associated to $K$ as described above.
The group ${\rm Aut}^{+}(S,\beta)$ corresponds to $N_{\Delta(a,b,c)}(K)/K$ and ${\rm Aut}(S,\beta)$ corresponds to $N_{\overline{\Delta}(a,b,c)}(K)/K$, where $N_{\Delta(a,b,c)}(K)$ and $N_{\overline{\Delta}(a,b,c)}(K)$ are the corresponding normalizers of $K$ in $\Delta(a,b,c)$ and $\overline{\Delta}(a,b,c)$, respectively.
\end{theorem}

As non-finite triangular groups admit proper and cocompact actions on both the real and hyperbolic planes, Propoposition \ref{pro:acts_properly} implies the following.

\begin{corollary}\label{unfin}
The ends space of the triangular group $\triangle(a,b,c)$, where $a^{-1}+b^{-1}+c^{-1} \leq 1$,  has one end. Finite triangular groups have no ends. 
\end{corollary}

\begin{remark}\label{remark:fines}
As a consequence of the above and the results in \cite{ARV}, 
the ends space of the automorphims group of a regular map has at most one end.
Let $\triangle(a,b,c)$ be a Fuchsian triangle group and $K\lhd \triangle(a,b,c)$ such that $[\triangle(a,b,c):K]=\infty$. Then the group $\triangle(a,b,c) /K$ has only one end (see \cite{Serre}*{Example 6.3.2, p. 60}), in particular, the surface $\mathbb{H}^2/ K$ has one end.
\end{remark}

%%%%%%%%%%%%%%%%%%%%%%%%%%%%%%%%%%%
%%%%%%%%%%%%%%%%%%%%%%%%%%%%%%%%%%%
\section{On the Loch Ness monster}
In this section we describe Riemann surface structures on the Loch Ness monster (LNM),  coming as: (i) homology covers of compact hyperbolic Riemann surfaces, (ii) quotients of certain infinite index subgroups of ${\rm PSL}_{2}({\mathbb Z})$, and (iii) affine plane curves (infinite hyperelliptic and superelliptic  curves).

%%%%%%%%%%%%%%%%%%%%%%%%%%%%%%%%%
\subsection{Hyperbolic structures from compact surfaces}
In this subsection, we observe that to the LNM can be endowed of infinitely many Riemann surface structures coming from infinite index characteristic subgroups of the fundamental group of a compact Riemann surface of genus at least two (see Proposition \ref{caracteristico}).

\begin{proposition}\label{caracteristico}
Let $\Gamma$ be a co-compact Fuchsian group of genus $g \geq 2$ and let $K$ be a characteristic subgroup of $\Gamma$ such that $[\Gamma:K]=\infty$. Then ${\mathbb H}^{2}/K$ is topologically the LNM. 
\end{proposition}
\begin{proof}
For each genus $g \geq 2$ there is a regular Belyi pair $(S,\beta)$, where $S$ is a compact Riemann surface of genus $g$. Let $G<{\rm Aut}(S)$ be the deck group of $\beta$, and
let $\Delta(a,b,c)$ be a Fuchsian triangular group such that 
$S/G={\mathbb H}^{2}/\Delta(a,b,c)$. Then there is a normal (torsion free) subgroup $F$ of the triangular group $\Delta(a,b,c)$ such that $S={\mathbb H}^{2}/F$ and $G=\Delta(a,b,c)/F$. Moreover, there exists an orientation preserving homeomorphism $h:{\mathbb H}^{2} \to {\mathbb H}^{2}$ conjugating $\Gamma$ to $F$. The homeomorphism $h$ induces an orientation preserving homeomorphism between ${\mathbb H}^{2}/K$ and ${\mathbb H}^{2}/hKh^{-1}$. As $hKh^{-1}$ is a characteristic subgroup of $F$, it is a normal subgroup of the triangular group $\Delta(a,b,c)$, so the result follows from Corollary \ref{unfin} and Remark \ref{remark:fines}.
\end{proof}

\subsubsection{The Loch Ness monster as a homology cover}
Let $K=\Gamma' \lhd \Gamma$ be the derived subgroup of a co-compact Fuchsian group $\Gamma$ of genus $g \geq 2$. Set $S_{\Gamma}={\mathbb H}/\Gamma$, a compact Riemann surface of genus $g$, and $S_{K}={\mathbb H}/K$, a non-compact Riemann surface (this is called the homology cover of $S_{\Gamma}$). As a consequence of Proposition \ref{caracteristico}, we obtain the following.

\begin{corollary}\label{homologia}
The homology cover of a compact Riemann surface of genus $g \geq 2$ is topologically equivalent to the LNM.
\end{corollary}

\begin{remark}
The above result can also be obtained as follows.
The group $G=\Gamma/K \cong {\mathbb Z}^{2g}$, has exactly one end (see \cite{ScottWall}*{Corollary 5.5}), and acts as a group of holomorphic automorphisms of $S_{K}$ such that $S_{\Gamma}=S_{K}/G$. 
Let $P \subset {\mathbb H}$ be a canonical fundamental polygon for $\Gamma$ (it has $4g$ sides and its pairing sides are identified by a set $A_{1}, \ldots, A_{g}, B_{1}, \ldots, B_{g} \in \Gamma$, of generators of $\Gamma$ such that $\prod_{j=1}^{g} [A_{j},B_{j}]=1$). The $\Gamma$-translates of $P$ induces a tessellation on ${\mathbb H}$ and it descends to a tessellation on $S_{K}$. The dual graph of such an induced tessellation is the Cayley graph of $G$, with respect to the induced $2g$ generators by the elements $A_{j}, B_{j}$. As ${\mathbb Z}^{2g}$, $g \geq 2$, has one end, it follows that the number of ends of $S_{K}$ must be one.
\end{remark}

It is well-known that the homology cover determines the compact Riemann surface, more precisely:

\begin{theorem}[\cite{Maskit:homology}*{p. 561}]
Let $\Gamma_{1}$ and $\Gamma_{2}$ be two co-compact Fuchsian groups. If $\Gamma'_{1}=\Gamma'_{2}$, then $\Gamma_{1}=\Gamma_{2}$. In particular, two compact Riemann surfaces, both of genus at least two, are isomorphic if and only if their homology covers are isomorphic.
\end{theorem}

\subsubsection{The homology cover of LNM}
Let $K=(\Gamma')' \lhd \Gamma$ be the double  derived subgroup of a co-compact Fuchsian group $\Gamma$ of genus $g \geq 2$. In the previous case, we have seen that $\mathbb{H}^2/\Gamma'$ is topologically equivalent to the LNM. As $K$ is characteristic subgroup of $\Gamma$, by Proposition  \ref{caracteristico}, we also have that $\mathbb{H}^2/ K$ is topologically equivalent to LNM. As a consequence, we obtain the following.

\begin{corollary}
The homology cover of the LNM  is topologically equivalent to the LNM.
\end{corollary}

\subsubsection{Another example of characteristic subgroups}

Another characteristic subgroup of $\Gamma$ is given by the subgroup $\Gamma^{k}$ generated by all the $k$-powers of its elements, where $k \geq 3$. 
In this case, \[
G=\Gamma/\Gamma^{k}
=\langle x_{1},\ldots, x_{g}, y_{1},\ldots, y_{g}: x_{1}^{k}=\cdots=x_{g}^{k}=y_{1}^{k}=\cdots=y_{g}^{k}=\prod_{j=1}^{g}[x_{j},y_{j}]=1\rangle
\]
and $\Gamma^{k}$ has infinite index in $\Gamma$. So, by the Proposition \ref{caracteristico} we obtain the following.

\begin{corollary}
Let $\Gamma$ be a torsion-free co-compact Fuchsian group of genus $g \geq 2$. For each $k \geq 3$, let $\Gamma^k$ be the characteristic subgroup generated by all the $k$-powers of the elements of $\Gamma$. Then ${\mathbb H}/\Gamma^{k}$ is topologically equivalent to the LNM.
\end{corollary}

%%%%%%%%%%%%%%%%%%%%%%
\begin{remark}[Dessins d'enfants on the Loch Ness monster]
Let $(S,\beta)$ be a hyperbolic regular Belyi pair of type $(a,b,c)$, where $S$ is a compact Riemann surface of genus $g \geq 2$.  This Belyi pair corresponds to a finite index torsion-free normal subgroup $\Gamma$ of the triangular group $\Delta(a,b,c)$, that is, $S={\mathbb H}/\Gamma$ and $\beta$ is induced by the normal inclusion $\Gamma \lhd \Delta(a,b,c)$. Let $K=\Gamma'$ be the derived subgroup of $\Gamma$. As seen above, the homology cover $S_{K}={\mathbb H}/K$ of $S$ is topologically equivalent to the LNM. 
As $K$ is characteristic subgroup of $\Gamma$ and the last is a normal subgroup of the triangular group $\Delta(a,b,c)$, then we obtain that  $K$ is also normal subgroup of it. 
The normal inclusion $K \lhd \Delta(a,b,c)$ induces a regular Belyi pair $(S_{K},\widetilde{\beta})$ of type $(a,b,c)$.  Then there is a regular unbranched covering $P:S_{K} \to S$, whose deck group is $\Gamma/K \cong {\mathbb Z}^{2g}$, such that $\widetilde{\beta}=\beta \circ P$. If, moreover, $S$ admits no anticonformal automorphism, then neither does $S_{K}$, in particular, if the Belyi pair $(S,\beta)$ is chiral, then the same happens with the Belyi pair $(S_{K},\widetilde{\beta})$. The above fact, together Proposition \ref{Proposition:Belyi-dessins}, asserts the existence of infinitely many regular dessins d'enfants (either chiral or reflexive) on the LNM.
\end{remark}

%%%%%%%%%%%%%%%%%%%%%%%%%%%%%%%%%%%
\subsection{Hyperbolic structures from subgroups of ${\rm PSL}_{2}({\mathbb Z})$} 
Let us consider the modular group ${\rm PSL}_{2}({\mathbb Z})$, which is generated by $F(z)=z+1$ and $E(z)=-1/z$. Inside this group there is  an index three non-normal subgroup $K_{0} \cong \mathbb{Z}\ast \mathbb{Z}_2$, generated by $A=F^{2}$ and $E$. A fundamental domain for $K_0$ is the geodesic triangle whose vertices are the points $-1$, $1$ and $\infty$ (see Figure \ref{F:Triangle_1}). The quotient orbifold $\mathbb{H}/ K_{0}$ is the punctured complex plane ${\mathbb C}\setminus\{\textbf{0}\}$  with exactly one cone point (say $1$) with cone order $2$.
\begin{center}
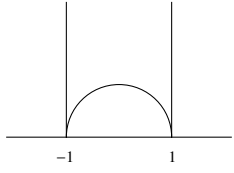
\begin{figure}[h!]
\begin{tikzpicture}[baseline=(current bounding box.north)]
\begin{scope}
    \clip (-1.5,0) rectangle (1.5,2.0);
    \draw (0,0) circle(0.7);
     \draw (-0.7,0) -- (-0.7,1.8);
      \draw (0.7,0) -- (0.7,1.8);
    \draw (-1.5,0) -- (1.5,0);
        \end{scope}
\node[below left= 1mm of {(-0.4,0)}] {\tiny{$-1$}};
\node[below left= 1mm of {(0.96,0)}] {\tiny{$1$}};
\end{tikzpicture} 
 \caption{\emph{Fundamental domain for $K_0$}}
   \label{F:Triangle_1}
\end{figure}
\end{center}

We will denote by $C_r(z)$ the half-circle with center $z \in\mathbb{C}$ and radius $r>0$. The M\"obius transformation $A^2E$ sends the half-circle $C_1(\textbf{0})$ onto $C_1(4)$ and, the  M\"obius transformation $A EA^{-3}$  sends the half-circle $C_1(6)$ onto $C_1(2)$. 

Note that 
\[
K_1=\langle A^4, A^{2}E, AEA^{-3}\rangle \cong \mathbb{Z}\ast \mathbb{Z}\ast \mathbb{Z},
\]
is a subgroup of index four of $K_0$.
A fundamental domain for $K_1$ is the  geodesic polygon with six sides whose vertices are the points $-1$, $1$, $3$, $5$, $7$ and $\infty$ (see Figure \ref{F:Triangle_2}), and 
the quotient orbifold $\mathbb{H}/ K_1$ is homeomorphic to the complex torus with two punctures.  We observe that the group $K_{1}$ is normalized by $A$, so it induces a holomorphic automorphism of order $4$ on the complex torus ${\mathbb H}/K_{1}$, such that each one of these two punctures is fixed. It follows that the compactification of the complex torus ${\mathbb H}/K_{1}$ is defined by the elliptic curve $y^{2}=x^{4}-1$.
\begin{center}
\begin{figure}[h!]
\begin{tikzpicture}[baseline=(current bounding box.north)]
\begin{scope}
    \clip (-1.5,0) rectangle (4.5,2.2);
    \draw (0,0) circle(0.5);
     \draw (1,0) circle(0.5);
     \draw (2,0) circle(0.5);
     \draw (3,0) circle(0.5);
     \draw (-0.5,0) -- (-0.5,2);
      \draw (3.5,0) -- (3.5,2);
    \draw (-1.5,0) -- (4.5,0);
    \draw[->] (0,1.82) -- (3,1.82);
        \node at (1.8,2.1) {\tiny{$A^4$}};
     \node at (2.5,1.47) {\tiny{$AEA^{-3}$}};
     \node at (0.6,1.47) {\tiny{$A^{2}E$}};
\end{scope}
\draw[color=blue!50, thick, ->] (2.7,0.53) arc (0:180:0.73);
\draw[color=red!60, thick, <-] (1.7,0.53) arc (0:180:0.73);
\node[below left= 1mm of {(-0.2,0)}] {\tiny{$-1$}};
\node[below left= 1mm of {(0.8,0)}] {\tiny{$1$}};
\node[below left= 1mm of {(1.8,0)}] {\tiny{$3$}};
\node[below left= 1mm of {(2.8,0)}] {\tiny{$5$}};
\node[below left= 1mm of {(3.8,0)}] {\tiny{$7$}};
\end{tikzpicture} 
 \caption{\emph{Fundamental domain for $K_1$}}
   \label{F:Triangle_2}
\end{figure}
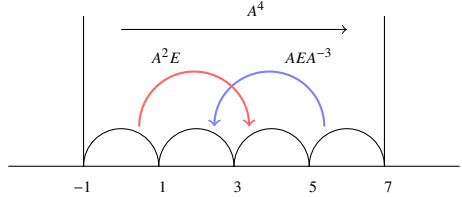
\end{center}
 
For each $n\in\mathbb{N}$, we define the subgroup 
\[
K_{n}=\langle A^{4n},\ A^{4l}A^2E A^{-4l},\ A^{4l} AEA^{-3} A^{-4l}:\ 1-n \leq l \leq n-1\rangle.
\]
\begin{remark}\label{r:geodesic_polygon}
From the definition above, the group $K_{n}$ is isomorphic to  
$\mathbb{Z}\ast \stackrel{4n-1}{\ldots} \ast \mathbb{Z}$, and has index $2n-1$ in $K_{1}$. A fundamental domain for $K_n$ is the geodesic polygon with $8n-2$ sides whose vertices are the points $\infty$ and the $2l-1$, with $-4(n-1) \leq l\leq 4n$ (see {\rm Figure \ref{F:Triangle_n}}), and the quotient orbifold  $\mathbb{H}/ K_n$ is homeomorphic to a surface of genus $2n-1$ with two punctures.

\begin{figure}[h!]
\begin{center}
\begin{tikzpicture}[baseline=(current bounding box.north)]

\begin{scope}
    \clip (-5.5,0) rectangle (7.5,1.8);
    \draw (0,0) circle(0.25);
     \draw (0.5,0) circle(0.25);
     \draw (1.0,0) circle(0.25);
     \draw (1.5,0) circle(0.25);
     \draw[dashed] (-0.25,0) -- (-0.25,2);
     \draw[dashed] (1.75,0) -- (1.7,2);
    \draw[dashed] (-5.5,0) -- (7.0,0);
      \draw (2.0,0) circle(0.25);
    \draw (2.5,0) circle(0.25);
     \draw (3.0,0) circle(0.25);
      \draw (3.5,0) circle(0.25);
      \draw[dashed] (3.75,0) -- (3.75,2);
        \node at (4.25,1.0) {$\cdots$};
      \draw[dashed] (4.7,0) -- (4.7,2);
        \draw (4.95,0) circle(0.25);
         \draw (5.45,0) circle(0.25);
          \draw (5.95,0) circle(0.25);
           \draw (6.45,0) circle(0.25);
           \draw (6.7,0) -- (6.7,2);
       
        \draw (-0.5,0) circle(0.25);
         \draw (-1.0,0) circle(0.25);
          \draw (-1.5,0) circle(0.25);
           \draw (-2.0,0) circle(0.25);
            \draw[dashed] (-2.25,0) -- (-2.25,2);
             \node at (-2.7,1.0) {$\cdots$};
            \draw[dashed] (-3.2,0) -- (-3.2,2);
            \draw (-3.45,0) circle(0.25);
            \draw (-3.95,0) circle(0.25);
            \draw (-4.45,0) circle(0.25);
            \draw (-4.95,0) circle(0.25);
            \draw (-5.20,0) -- (-5.20,2);                            
\end{scope}
\draw[color=blue!50, thick, ->] (1.4,0.3) arc (0:180:0.4);
\draw[color=red!60, thick, <-] (0.9,0.3) arc (0:180:0.4);
\draw[color=blue!50, thick, ->] (3.4,0.3) arc (0:180:0.4);
\draw[color=red!60, thick, <-] (2.9,0.3) arc (0:180:0.4);
\draw[color=blue!50, thick, ->] (-0.6,0.3) arc (0:180:0.4);
\draw[color=red!60, thick, <-] (-1.1,0.3) arc (0:180:0.4);
\draw[color=blue!50, thick, ->] (6.4,0.3) arc (0:180:0.4);
\draw[color=red!60, thick, <-] (5.9,0.3) arc (0:180:0.4);
\draw[color=blue!50, thick, ->] (-3.5,0.3) arc (0:180:0.4);
\draw[color=red!60, thick, <-] (-4.0,0.3) arc (0:180:0.4);
\node[below left= 1mm of {(0,0)}] {\tiny{$-1$}};
\node[below left= 1mm of {(2.1,0)}] {\tiny{$7$}};
\node[below left= 1mm of {(4.2,0)}] {\tiny{$7+8$}};
\node[below left= 1mm of {(7.5,0)}] {\tiny{$7+(n-1)\cdot8$}};
\node[below left= 1mm of {(-1.8,0)}] {\tiny{$-1-8$}};
\node[below left= 1mm of {(-4.3,0)}] {\tiny{$-1-(n-1)\cdot 8$}};
\end{tikzpicture} 

 \caption{\emph{Fundamental domain for $K_n$}}
   \label{F:Triangle_n}
\end{center}
\end{figure}
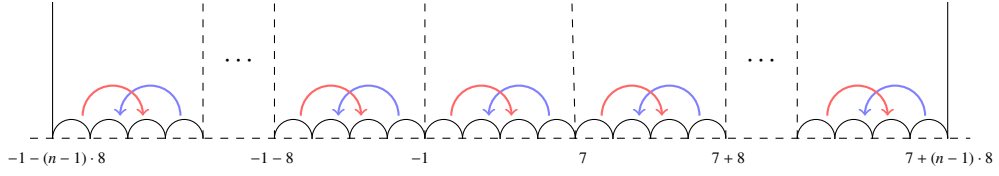
\end{remark}

The group 
$$K_{\infty}=\langle A^{4l} A^{2}E A^{-4l}, A^{4l} AEA^{-3} A^{-4l}: l \in {\mathbb Z}\rangle$$
is an infinite index subgroup of $K_{0}$, which is normalized by $A^{4}$, and such that the quotient orbifold
\begin{equation}
S={\mathbb H}/K_{\infty}
\end{equation}
is topologically equivalent to the Loch Ness monster (see \cite{AR2020}*{Theorem 1.3}).

%%%%%%%%%%%%%%%%%%%%%%%%%%%%%%%%%%%%%%%%%%%
\subsection{Riemann surface structures coming from affine plane curves on the Loch Ness monster}

Let $U$ be a non-empty connected open subset of $\mathbb{C}^{2}$, and let $F:U \subset {\mathbb C}^{2} \to {\mathbb C}$ be a non-constant holomorphic map. For each point $p\in F(U)$, the set $S_{F}(p):=\{(z,w) \in U: F(z,w)=p\}$ is called {\it affine plane curve}. When $U={\mathbb C}^{2}$ and $F$ is a polynomial map, then usually we talk of an {\it algebraic set}.

\begin{remark}
The affine plane curve $S_{F}(p)$ is a closed subset of $U$, because this is the inverse image of $p$ under the continuous map $F$.
\end{remark}

If $p$ is a regular value for $F$, then, by the Implicit Function Theorem (see \emph{e.g.}, \cite{KraPa}, \cite{Miranda}*{p. 10, Theorem 2.1}),  the affine plane curve $S_{F}(p)$ is a Riemann surface.

\begin{example}
Let $F:{\mathbb C}^{2} \to {\mathbb C}$ be the holomorphic map  given by $F(z,w)=z e^{w}$. As $\frac{\partial F}{\partial z} (z,w) = e^{w}\neq 0$, we obtain that  $S_{F}(1)=\{(z,w): ze^w=1\}$ is a Riemann surface.

\end{example}

%%%%%%%%%
\subsubsection{\bf Infinite superelliptic curves}
Let us consider a sequence of different complex numbers $(z_k)_{k\in\mathbb{N}}$, such that $\lim\limits_{k\to\infty}|z_k|=\infty$, and a sequence $(m_{k})_{k\in\mathbb{N}}$  of positive integers. 
 By the Weierstrass theorem (see \emph{e.g.}, \cite{Palka}*{p. 498}) there exists a meromorphic function $f:{\mathbb C} \to \widehat{\mathbb C}$ whose zeroes are given by the points $z_1,z_2,\ldots$, each $z_{k}$ of order $m_k$. Moreover, $f$ is uniquely determined (up multiplication) by  a zero-free entire map (for example $e^{z}$). Such functions $f$ admit the representation 
\begin{equation}\label{eq:T_Weierstrass}
f(z)=g(z)z^{m_{0}}\prod_{k=1, z_{k} \neq 0}^{\infty}\left(1-\frac{z}{z_k}\right)^{m_k}E_k(z),
\end{equation}
where $g$ is a zero-free entire function ($m_{0}=0$ if $z_{k} \neq 0$ for every $k\in\mathbb{N}$; in the other case, $m_{0}$ is the corresponding $m_{k}$ for $z_{k}=0$), 
and $E_{k}(z)$ is a function of the form
\[
E_{k}(z)=\exp \left[\sum_{s=1}^{d(k)}\frac{1}{s}\left(\frac{z}{z_{k}}\right)^{s} \right],
\]
for a suitably large non-negative integer $d(k)$.

Now, if we consider the holomorphic function $F:{\mathbb C}^{2} \to {\mathbb C}$ given by $F(z,w)=w^{n}-f(z)$, with $n \geq 2$ integer, then we obtain the affine plane curve 
\[
S(f)=\{(z,w) \in {\mathbb C}^{2}: w^{n}=f(z)\}.
\]
If $m_{k}=1$ for all $k\in\mathbb{N}$ (see the equation  (\ref{eq:T_Weierstrass})), then such affine curve $S(f)$ is a Riemann surface, called {\it infinite superelliptic curve}. If moreover, $n=2$, the affine curve $S(f)$ is known as  \emph{infinite hyperelliptic curve}.

\begin{remark}\label{r:covering}
The projection map $\pi_z:S(f)\to \mathbb{C}$, $(z, w)\mapsto z$, satisfies the following properties.
\begin{enumerate}
\item It is a branched covering map whose branch points are given by the sequence $(z_k)_{k\in\mathbb{N}}$. 

\item For each $z\in \mathbb{C}\setminus\{z_k:k\in\mathbb{N}\}$, the fiber $\pi_{z}^{-1}(z)$ consists of $n$ elements.

\item It is a proper map, that is, the inverse image of any compact subset $K$ of $\mathbb{C}$ is also a compact subset of $S(f)$ (because $\pi_{z}^{-1}(K)$ is a closed subset of the compact $K\times \left(h^{-1} [ f(K)] \right)$, where $h$ is the complex  $n$-power map $h(w)=w^n$).
\end{enumerate}
\end{remark}

The following result describes the topology type of an infinite superelliptic curve. 

\begin{theorem}\label{t:infinite_hyperelliptic_curve}
 The infinite superelliptic curve $S(f)$ is a connected Riemann surface homeomorphic to the Loch Ness monster.
\end{theorem}

\begin{proof}
First, we shall prove that $S(f)$ is path-connected. The sequence $(z_k)_{k\in\mathbb{N}}$ is ordered as follows: (i) $ |z_k|< |z_{k+1}|$ and, (ii) if $|z_k|=|z_{k+1}|$, then $0\leq \arg(z_k)<\arg(z_{k+1})<2\pi$. As the complex plane $\mathbb{C}$ is path-connected, we can consider a simple smooth arc 
\begin{equation}\label{eq:path_on_complex_plane}
\gamma:[0,\infty)\to\mathbb{C},    
\end{equation} 
such that there are real numbers $0=x_1<x_2<\ldots$ satisfying $\gamma(x_k)=z_k$, for each $k\in\mathbb{N}$. The inverse image $G(\gamma):=\pi_{z}^{-1}(\gamma[0,\infty))$ is called {\it a  spine of} $S(f)$ {\it associated to $\gamma$}. From Remark \ref{r:covering} we claim that the spine $G(\gamma)$ is a closed and connected subset of $S(f)$. 
 
Given that $\pi_z: S(f)\to \mathbb{C}$ is a branched covering (Remark \ref{r:covering} (1) and (2)), if we take a point $(z_0,w_0)\in S(f)$ and a path $\beta$ in $\mathbb{C}\setminus \{z_k:k\in\mathbb{N}\}$, such that one of its end points is $z_0$ and the other one end point belongs to $\gamma[0,\infty)$, then there exists  a lifting path $\tilde{\beta}$ of $\beta$ on $S(f)$, such that one of its end points is $(z_0,w_0)$ and the other end point belongs to the spine $G(\gamma)$. This shows that $S(f)$ is path-connected. 

In order to prove that $S(f)$ is homeomorphic to the Loch Ness monster, we must prove that $S(f)$ has only one end and infinite genus.

\emph{The unique end of $S(f)$.} Given a compact subset $K\subset S(f)$ we shall prove that there is a compact subset $K^{'}\subset S(f)$ such that $K\subset K^{'}$ and $S(f)\setminus K^{'}$ is connected. The image $\pi_{z}(K)$ is a compact subset of the complex plane $\mathbb{C}$. As the complex plane $\mathbb{C}$ has only one end, then there exists a real number $r>0$, such that the closed ball $\overline{B_r(\textbf{0})}$ contains to the compact $\pi_{z}(K)$ and $\mathbb{C}\setminus \overline{B_{r}(\textbf{0})}$ is connected. We can suppose without loss of generality that there exists $N\in\mathbb{N}$ such that the intersection $\overline{B_{r}(\textbf{0})}\cap \{z_k :k\in\mathbb{N}\}=\{z_1,\ldots, z_{N-1}\}$. Moreover, the intersection $\overline{B_r(\textbf{0})}\cap \gamma[x_N,\infty)=\emptyset$ (see equation (\ref{eq:path_on_complex_plane})). The inverse image $G:=\pi_{z}^{-1}(\gamma[x_N,\infty))$ is a closed connected subset of the spine $G(\gamma)$, and given that the projection map $\pi_z$ is a proper map (Remark \ref{r:covering} (3)), the inverse image $K^{'}:=\pi_{z}^{-1}\left(\overline{B_r(\textbf{0})}\right)$ is a compact subset of $S(f)$. Moreover, $K\subset K^{'}$ and $G\cap K^{'}=\emptyset$. Using the same ideas described in the proof of the path-connected of $S(f)$, for each point $(z_0,w_0)\in S(f)\setminus K^{'}$ we can found a path $\beta$ in $S(f)\setminus K^{'}$ having as one of its end points $(z_0,w_0)$ and the other end point belongs to $G$. Hence, $S(f)\setminus K^{'}$ is path-connected. Thus, we conclude that the surface $S(f)$ has only one end.

\emph{The surface $S(f)$ has infinite genus.} For each $d\in\mathbb{N}$, let $S_{d}$ be the compact Riemann surface associated to the algebraic curve
\[
v^n=\prod_{k=1}^{dn}(u-z_k).
\]
By the Riemann-Hurwitz formula,
the compact Riemann surface $S_{d}$ has genus $g=\frac{n}{2}(d(n-1)-2)+1$.  

As the complex plane $\mathbb{C}$ is an $\sigma$-compact space, we can take an increasing sequence of compact connected subsets $K_1\subset K_2\subset \ldots$ such that $\mathbb{C}=\bigcup\limits_{d\in\mathbb{N}}K_{d}$, and for each $d\in\mathbb{N}$, $K_{d}\cap \{z_k:k\in\mathbb{N}\}= \{z_1,\ldots,z_{dn}\}$. By the Riemann-Hurwitz formula the subsurface 
\[
S_{d}(K_{d}):=\left\{(u,v)\in\mathbb{C}^2: v^n=\prod_{k=1}^{dn}(u-z_k),  u\in K_{d}\right\}\subset S_{d}, 
\]
is homeomorphic to a $g$-torus with one boundary. 

Now, we define an embedding map $\varphi$ from $S_{d}(K_{d})$ to $S(f)$ using the projection maps $\pi_{u}:S_{d}(K_{d}) \to K_{d}$  and $\pi_z: S(f)\to \mathbb{C}$. We fix a point $\tilde{z}\in K_{d}-\{z_1,\ldots,z_{nd}\}$, then there are different points $p_1,\ldots,p_n$ in the fiber  $\pi_{u}^{-1}\left(\tilde{z}\right)\subset S_{d}(K_{d})$. Similarly, there are different points $q_1,\ldots,q_n$ in the fiber $\pi_{z}^{-1}(\tilde{z})$. Then we define 
\[
\varphi(p_i)=q_i, \text{ for each } i\in \{1,\ldots,n\}.
\]
For each point $s$ in  $S_{d}(K_{d})$, we take a path $\gamma$ in $K_{d}-\{z_1,\ldots,z_{nd}\}$ having ends points $\tilde{z}$ and $\pi_u(s)$. Then there exists a lifting $\gamma_i$ of $\gamma$ on $S_{d}(K_{d})$ such that its ends point are $s$ and $p_i$, for any $i\in \{1,\ldots,n\}$. Similarly, there exists a lifting $\tilde{\gamma}_i$ of $\gamma$ on $S(f)$ such that its ends points are $q_i$ and any $t$ in $S(f)$. Thus, we define $\varphi(s)=t$. Then by construction $\varphi$ is a well define injective map, and  $\pi_{z}\circ \varphi(s)=\pi_{u}(s)$, for each $s\in S_{d}(K_{d})$.

Next, we must prove that the map $\varphi$ is continuous. Let $s$ be a point of $S_{d}(K_{d})$ and let $U$ be an open subset of $\varphi(s)\in S(f)$. As the projection map $\pi_{z}$ is an open map, then $\pi_z(U)$ is an open subset of the complex plane $\mathbb{C}$  containing the point $\pi_{u}(s)$. By the continuously of $\pi_u$ there exists an open $V$ of $s$ such that $\pi_{u}(V)\subset \pi_{z}(U)\cap K_{d}$.  It is easily shown that $\varphi(V)\subset U$. The map $\varphi$ is closed because is a continuous map from a compact space into a Hausdorff space (see \cite{Dugu}*{Theorem 2.1, p. 226}). Hence $f$ is an embedding. This fact implies that for each $d\in \mathbb{N}$, there exists a subsurface of $S(f)$ having genus $g=\frac{n}{2}(d(n-1)-2)+1$.
\end{proof}

Let us observe that the surface $S(f)$ admits the holomorphic automorphism 
$$\varphi_{f}:S(f) \to S(f): (z,w) \mapsto (z,e^{2 \pi i/n}w).$$ 
The holomorphic branched covering map $\pi_{z}:S(f) \to {\mathbb C}$ is Galois with deck group $G_{f}:=\langle \varphi_{f} \rangle \cong {\mathbb Z}_{n}$.

Now, we consider $(z_k)_{k\in\mathbb{N}}$ and $(z'_k)_{k\in\mathbb{N}}$ sequences of complex number such that  $\lim\limits_{k\to \infty}|z_k|=\infty=\lim\limits_{k\to \infty}|z'_k|$. Let $f$ and $g$ be the entire maps from Weierstrass's theorem having as simple zeros the points $z_1,z_2,\ldots$ and $z'_1,z'_2,\ldots$, respectively.
\begin{theorem}
If $n=2$, i.e., for the hyperelliptic case,
the pairs $(S(f),G_{f})$ and $(S(g),G_{g})$ are biholomorphically equivalent if and only if there exists a  holomorphic automorphism of complex plane $\mathbb{C}$ carrying the zeros of $f$ onto the zeros of $g$.
\end{theorem}
\begin{proof}
Note that the orbifold $S(f)/G_{f}$ is given by the complex plane $\mathbb{C}$ with conical points of order two at the values $z_{k}$. A Fuchsian group $\Gamma$ uniformizing it is an infinite free product of elliptic elements of order two. The group $\Gamma$ has a unique index two subgroup $\Gamma_{0}$ which is torsion-free. This group $\Gamma_{0}$ provides the uniformization of $S(f)$ and the branched covering $\pi_{z}:S(f) \to {\mathbb C}$ is induced by the pair $(\Gamma, \Gamma_{0})$. It follows that every biholomorphim of the orbifolds $S(f)/G_{f}$ and $S(g)/G_{g}$ lifts to an automorphism between $S(f)$ and $S(g)$ conjugating $G_{f}$ to $G_{f}$. The converse is clear.
\end{proof}

\begin{example}\label{example:7.1}
Given that the zeros of the sine and the cosine complex maps, 
\[
f(z)=\sin(\pi z)=\pi z\prod_{n=1}^{\infty}\left(1-\frac{z^2}{n^2}\right) \text{  and, }  g(z)=\cos(\pi z)=\prod_{n=1}^{\infty}\left(1-\frac{4z^2}{(2n-1)^2}\right)
\]
differ by a translation, then the infinite hyperelliptic curves 
\[
S(f)=\left\{(z,w)\in\mathbb{C}^2: w^2=f(z)\right\}
\text{ and, } S(g)=\left\{(z,w)\in\mathbb{C}^2: w^2=g(z)\right\}
\]
are biholomorphic and topologically equivalent to the Loch Ness monster.
\end{example}

\subsubsection{\bf Moduli space of infinite hyperelliptic curves}
The space $\ell^{\infty}$ conformed by all bounded sequence of complex numbers $(z_k)_{k\in\mathbb{N}}$ equipped with the norm 
\[
|(z_k)_{k\in\mathbb{N}}|_{\infty}=\sup\limits_{k\in\mathbb{N}} |z_k|<\infty
\]
is a complete space, and the set $c_0$ conformed by all sequences of complex number with limit $\textbf{0}$ is a closed subset of $\ell^{\infty}$. The set $c_{\infty}$ conformed by all sequence of complex numbers $(z_k)_{k\in\mathbb{N}}$ such that $\lim\limits_{k\to \infty}|z_n|=\infty$ can be identified with $c_0$ using the map $z\mapsto\frac{1}{z}$. So, two infinite hyperelliptic curves $S(f)$ and $S(g)$, such that the sequences $(z_k)_{k\in\mathbb{N}}$ and $(z'_k)_{k\in\mathbb{N}}$ are the simple zeros of $f$ and $g$ respectively, are biholomorphic if there exists affine map $z\mapsto az+b$ with $a,b\in \mathbb{C}$ and $a\not=\textbf{0}$ such that $z_k\mapsto z'_k$, for each $k\in\mathbb{N}$.
Then, by the above, the group $G=\{w\mapsto \frac{w}{a+bw}:a,b\in\mathbb{C}, a\neq \textbf{0}\}$ acts on $c_0$ by $(w_k)_{k\in\mathbb{N}}\mapsto \left(\frac{w_k}{a+bw_k}\right)_{k\in\mathbb{N}}$ with $w_k=\frac{1}{z_k}$. Thus, we have obtained the following fact.

\begin{theorem}
The moduli space of infinite hyperelliptic curves is given by $c_0/ G$.
\end{theorem}

%\subsection{}
In the case of hyperelliptic curves $S(f)$ of finite genus $g\geq 2$, the projection map $\pi_z$ is unique up to fractional linear transformations \cite[p. 95]{Far}. In the case of infinite hyperelliptic curve, it is natural to ask.

\begin{question}[Open]
Is  the projection map $\pi_z$ unique up to automorphism of $\mathbb{C}$?
\end{question}

\section*{Acknowledgements}

Camilo Ram\'irez Maluendas was partially supported by UNIVERSIDAD NACIONAL DE COLOMBIA, SEDE MANIZALES. He has dedicated this work to his beautiful family: Marbella and Emilio, in appreciation of their love and support.

%%%%%%%%%%%%%%%%%%%%%%%%%%%%%%%%%%%%%%%%%%%
%%%%%%%%BIBLIOGRAPHY
%%%%%%%%%%%%%%%%%%%%%%%%%%%%%%%

\end{document}